\documentclass[aip,graphicx]{revtex4-1}

\usepackage{enumerate}
\usepackage{dsfont}
\usepackage{amsfonts}
\usepackage{amssymb}
\usepackage[latin1]{inputenc}
\usepackage[T1]{fontenc}
\usepackage{graphicx}
\usepackage{mathtools}
\usepackage[dvipsnames]{xcolor}
\usepackage{bbm}
\usepackage{soul}
\usepackage[colorlinks=true,urlcolor=darkred,unicode]{hyperref}

\definecolor{darkred}{RGB}{153,0,0}
\usepackage[colorlinks=true]{hyperref}
\usepackage{amsthm}

\newcommand{\R}{\mathbb R}

\newcommand{\I}{\mathbb I}
\newcommand{\T}{\mathbb T}
\newcommand{\D}{\text{d}}

\newcommand{\db}{\bar\partial}

\newtheorem{prop}{Proposition}
\newtheorem{cor}[prop]{Corollary}
\newtheorem{lem}[prop]{Lemma}
\newtheorem{theor}[prop]{Theorem}

\newtheorem{defi}[prop]{Definition}
\newtheorem{ex}[prop]{Example}

\newcommand{\bluetext}[1]{{#1}}

\draft 

\begin{document}


\title{Berwald $m$-Kropina Spaces of Arbitrary Signature: Metrizability and Ricci-Flatness} 



\author{Sjors Heefer}
\email{s.j.heefer@tue.nl}
\affiliation{Department of Mathematics and Computer Science, Eindhoven University of Technology, Eindhoven, The Netherlands}
%
%


\date{\today}

\begin{abstract}
The (pseudo-)Riemann-metrizability and Ricci-flatness of Finsler spaces with $m$-Kropina metric $F = \alpha^{1+m}\beta^{-m}$ of Berwald type are investigated. We prove that the affine connection \bluetext{of} $F$  can locally be understood as the Levi-Civita connection of some (pseudo-)Riemannian metric if and only if the Ricci tensor of the canonical affine connection is symmetric. We also obtain a third equivalent characterization in terms of the covariant derivative of the 1-form $\beta$. We use these results to classify all locally metrizable $m$-Kropina spaces whose 1-forms have a constant causal character. In the special case where the first de Rham cohomology group of the underlying manifold is trivial (which is true of simply connected manifolds, for instance), we show that global metrizability is equivalent to local metrizability and hence, in that case, our necessary and sufficient conditions also characterize global metrizability. In addition, we further obtain explicitly all Ricci-flat, locally metrizable $m$-Kropina metrics in $(3+1)$D whose 1-forms have a constant causal character. In fact, the only possibilities are essentially the following two: either $\alpha$ is flat and $\beta$ is $\alpha$-parallel, or $\alpha$ is a pp-wave and $\beta$ is $\alpha$-parallel.
%
\end{abstract}



\maketitle 



%
%

%

\tableofcontents
\newpage
\section{Introduction}\label{sec:intr}

\noindent A Berwald space is a Finsler space for which the canonical nonlinear connection reduces to a linear connection on the tangent bundle $TM$. This affine connection is necessarily torsion-free. The natural question thus arises whether this affine connection arises as the Levi-Civita connection of some pseudo-Riemannian metric. Simply put, is (the affine connection on) every Berwald space metrizable? In the classical positive definite setting, this is indeed the case, as was first shown in (Szab\'o, 1981)\cite{Szabo}.

The proofs of this result \cite{Szabo, Vincze2005} rely on procedures such as averaging \cite{CrampinAveraging} over the indicatrix, 
for which it is essential that the Finsler metric $F$ is sufficiently smooth and defined everywhere on the slit tangent bundle $TM_0$. In the case of Finsler metrics of indefinite signature, however, the domain $\mathcal A\subset TM$ where $F$ is defined is typically only a proper subset of $TM_0$ and hence the proofs do not extend to this case. And even in the special case that $\mathcal A = TM_0$, the fact that the indicatrix is not compact in indefinite signatures poses problems. 

It was shown in (Fuster et al., 2020)\cite{Fuster_2020} that Szab\'o's metrization theorem indeed does not generalize to the more general setting. In (Heefer et al., 2023)\cite{Heefer2023mKropNull}, the situation was further investigated and necessary and sufficient conditions for local metrizability were obtained in the specific case of $m$-Kropina metrics with closed null 1-form. More recently, the metrizability of \bluetext{(spatially)} spherically symmetric Berwald \bluetext{spacetimes} was investigated in (Voicu et al., 2024)\cite{voicu2024metrizability}.

Here we extend the analysis of (Heefer et al., 2023)\cite{Heefer2023mKropNull} to arbitrary $m$-Kropina metrics, i.e. Finsler metrics of the form $F = \alpha^{1+m}\beta^{-m}$, with no restriction on the 1-form. The culprit behind all counterexamples to Szab\'o's theorem obtained in (Heefer et al., 2023)\cite{Heefer2023mKropNull} was the fact that the Ricci tensor constructed from the affine connection---the `affine Ricci tensor'---is in general not symmetric. This property---symmetry of the affine Ricci tensor---is in general a necessary condition for local metrizability, and for $m$-Kropina metrics with closed null 1-form it was shown to be a \bluetext{sufficient} condition as well\cite{Heefer2023mKropNull}. 

Here we show that this results extends to all $m$-Kropina metrics, i.e. we prove that any $m$-Kropina metric is locally metrizable if and only if its affine Ricci tensor is symmetric. Additionally, we further obtain a third necessary and sufficient condition for local metrizability in terms of the covariant derivative of the 1-form $\beta$. We use these results to classify all locally metrizable $m$-Kropina spaces whose 1-forms have a constant causal character. 

We also briefly touch upon the issue of global metrizability. We do not obtain a conclusive characterization of global metrizability in general, but in the special case where the first de Rham cohomology group of the underlying manifold is trivial (which is true of simply connected manifolds, for instance), we show that global metrizability is equivalent to local metrizability. Hence, in this case, our necessary and sufficient conditions for local metrizability also characterize global metrizability. 

Finally, as an application of our results established thus far, we obtain explicitly all Ricci-flat, locally metrizable $m$-Kropina metrics in $(3+1)$D whose 1-forms have a constant causal character. Here we note that an $m$-Kropina space is locally metrizable and Ricci-flat \bluetext{(in the standard sense)} if and only if it is \textit{affinely} Ricci-flat---meaning that its affine Ricci tensor \bluetext{(rather than the standard Finsler-Ricci tensor, its symmetrization)} vanishes. Affine Ricci-flatness turns out to be a  rather restrictive property. In fact, the only possibilities are essentially the following two: either $\alpha$ is flat and $\beta$ is $\alpha$-parallel, or $\alpha$ is a pp-wave and $\beta$ is $\alpha$-parallel. 

In the literature, $m$-Kropina metrics have been used for modeling the geometry of spacetime in Finsler extensions of Einstein's general theory of relativity. \cite{HeeferPhdThesis,Fuster:2018djw,Fuster:2015tua,Kouretsis:2008ha,Cohen:2006ky,Gibbons:2007iu}. Indeed $m$-Kropina metrics have been found to yield exact solutions\cite{Heefer_2023_Finsler_grav_waves,Fuster:2018djw,Fuster:2015tua} to the Pfeifer and Wohlfarth's Finslerian extension of Einstein's field equation\cite{Pfeifer:2011xi,Pfeifer:2013gha,Hohmann_2019}; see (Heefer, 2024)\cite{HeeferPhdThesis} for a recent review. Our results put strict bounds on the remaining possibilities for solutions of Berwald $m$-Kropina type to this field equation.

Many of the results obtained in this article were first described in the author's PhD thesis (Heefer, 2024)\cite{HeeferPhdThesis}.

\section{Prerequisites}\label{sec:Basics}
\noindent Below we introduce our notational conventions and we briefly recall several relevant notions in Finsler geometry. For a more in-depth treatment, the reader is referred to (Heefer, 2024)\cite{HeeferPhdThesis}.

\subsection{Some notational conventions}
\noindent In what follows, we will often work in local coordinates. Given a smooth manifold $M$ of dimension $n$ we will typically assume that some chart $\varphi:U\subset M\to \R^n$ is provided, and we will effectively identify any $p\in U$ with its image $x = (x^1,\dots,x^n)  \coloneqq\varphi(p)\in\R^n$ under $\varphi$. If $p\in U$ then each $X_p\in T_pM$ can be written as $X_p = y^i\partial_i\big|_p \eqqcolon y^i\partial_i$ (we will often be sloppy and suppress the dependence on the point $p$), where the tangent vectors $\partial_i \coloneqq \frac{\partial}{\partial x_i}$  make up the coordinate basis of $T_pM$. This decomposition provides natural local coordinates on the tangent bundle $TM$ via the chart
\begin{align}
\tilde\varphi: TU \to \R^n\times\R^n,\qquad \tilde\varphi(p,Y) = (\varphi(p),y^1,\dots,y^n)\eqqcolon (x,y),
\end{align}
where $TU$ is given by
\begin{align}
    TU = \bigcup_{p\in U} \left\{p\right\}\times T_p M\subset TM.
\end{align}
These local coordinates on $TM$ in turn provide a natural basis $\{\partial_i,\bar\partial_i\}_{i=1,\dots,n}$ of its tangent space $T_{(x,y)}TM$ at $(x,y)$, where we define
\begin{align}
\partial_i \coloneqq \frac{\partial}{\partial x^i}, \qquad\quad  \bar{\partial}_i\coloneqq \frac{\partial}{\partial y^i}.
\end{align}

\noindent Finally, we will use the notation $\D x^i\D x^j$ for the symmetrized tensor product of 1-forms, i.e. $\D x^i\D x^j \equiv \tfrac{1}{2}(\D x^j\otimes \D x^j + \D x^j \otimes \D x^j)$, and whenever we work specifically in Lorentzian signature we will adhere to the sign convention $(-,+,\dots,+)$.

\subsection{Basic notions in Finsler geometry}
\noindent For our purposes, a Finsler space will be a triple $(M,\mathcal A, F$), where $M$ is a smooth manifold, $\mathcal A$ a conic subbundle of $TM$, and $F$ is a Finsler metric on $\mathcal A$, defined as follows.

\begin{defi}
A conic subbundle of $TM$ is an open subset $\mathcal A\subset TM$ which is conic in the sense that $(x,\lambda y)\in\mathcal A$ for any $(x,y)\in\mathcal A$ and $\lambda>0$, and which satisfies $\pi(\mathcal A) = M$, where $\pi:TM\to M$ is the canonical projection of the tangent bundle. The latter property says that the fibers $\mathcal A_{x}$ of $\mathcal A$ are nonempty.
\end{defi}

\begin{defi}\label{def:Finsler_metric}
A Finsler metric on a conic subbundle $\mathcal A$ is a smooth map $F:\mathcal A\to [0,\infty)$ such that
\begin{itemize}
	\item $F$ is positively homogeneous of degree one:
	\begin{align}\label{eq:F_homogeneity}
	F(x,\lambda y) =\lambda F(x, y)\,,\quad \forall \lambda>0\,;
	\end{align}
	\item The \textit{fundamental tensor}, with components $g_{ij} = \db_i\db_j \left(\frac{1}{2}F^2\right)$, is nondegenerate.
\end{itemize}
\end{defi} 
\noindent The length of a curve $\gamma: (a,b) \to M$ (whose lift stays within $\mathcal A$) is then defined as
\begin{align}\label{eq:lengt_functional}
\ell(\gamma)=\int_a^b  F(\gamma(\lambda),\dot{\gamma}(\lambda))\,\D \lambda, \qquad \qquad\dot{\gamma}=\frac{\D\gamma}{\D\lambda}.
\end{align}
We will sometimes omit the specification of $\mathcal A$ and say, loosely speaking, that $F$ is a Finsler metric on $M$. We say that $F$ is positive definite if $g_{ij}$ is positive definite, and more generally that $F$ has a certain signature if $g_{ij}$ has that signature.

The connection coefficients of the canonical torsion-free, metric-compatible homogeneous nonlinear connection are given by
\begin{align}\label{eq:nonlinear_connection_explicit}
N_i^j = \tfrac{1}{4}\bar\partial_i\left[g^{jk}\left(y^m\partial_m\bar\partial_k L - \partial_k L\right)\right].
\end{align}
This nonlinear connection yields a smooth decomposition of the tangent bundle $T\mathcal A$ into a horizontal and a vertical subbundle\footnote{For alternative equivalent definitions, see (Javaloyes et al., 2022)\cite{javaloyes2023einsteinhilbertpalatini} and references therein.},
\begin{align}
T\mathcal A = H\mathcal A \oplus V\mathcal A,
\end{align}
where $\oplus$ denotes the Whitney sum of vector bundles. 
The vertical subbundle $V\mathcal A$ is canonically defined as
\begin{align}
V\mathcal A = \ker (\D\pi) =\text{span}\left\{\bar\partial_i\right\}.
\end{align}
The horizontal subbundle, on the other hand, is defined by the nonlinear connection as
\begin{align}
H\mathcal A = \text{span}\left\{\delta_i\right\},\qquad \delta_i\equiv \partial_i - N^j{}_i\db_j,
\end{align}
where we have also introduced the horizontal derivatives $\delta_i$.

The curvature tensor of the nonlinear connection has components
\begin{align}\label{eq:finsll_curvv}
    R^k_{ij} &= \delta_i N^k_j - \delta_j N^k_i,
\end{align}
and from this we can further define the following two quantities:
\begin{align}\label{eq:Ricci_defs}
    \text{Ric} = R^i{}_{ij}y^j,\qquad\quad R_{ij} = \tfrac{1}{2}\db_i \db_j\text{Ric}.
\end{align}
In the literature, $R_{ij}$ is often just referred to as the Ricci tensor, but we will give it the name \textit{Finsler-Ricci tensor} to distinguish it from the \textit{affine Ricci tensor} that will be introduced momentarily for Berwald spaces. Similarly, we will refer to $\mathrm{Ric}$ as the \textit{Finsler-Ricci scalar}. From homogeneity and Euler's theorem, it follows that the second relation in \eqref{eq:Ricci_defs} can be inverted as
\begin{align}
    \text{Ric} = R_{ij} y^i y^j.
\end{align}
Ric and $R_{ij}$ thus contain the same information, and in particular, we have \mbox{$\text{Ric}=0$} if and only if $R_{ij}=0$, in which the Finsler space is said to be \textit{Ricci-flat}.

\subsection{Berwald spaces}\label{sec:Berwald}

\noindent A Finsler space is said to be of \textit{Berwald} type, or simply Berwald, if the canonical nonlinear connection reduces to a (necessarily smooth) linear connection on $TM$, or in other words, an affine connection on the base manifold $M$, meaning that the connection coefficients are of the form $N^k_i = \Gamma^k_{ij}(x)y^j$ for certain coefficients $\Gamma^k_{ij}(x)$. In this case, the functions $\Gamma^i_{jk}$ are precisely the Christoffel symbols of said affine connection on $M$. We will refer to this affine connection as the associated affine connection, or simply \textit{the affine connection} on the Berwald space. It follows from the torsion-freeness of the nonlinear connection that the affine connection is torsion-free as well.

The curvature tensors \eqref{eq:finsll_curvv}-\eqref{eq:Ricci_defs} of a Berwald space can be written as
\begin{align}
\label{eq:symm_ricci_Berwald}
R^k_{ij} = \bar  R^k{}_{\ell ij} (x)y^\ell, \,\,\,\,\,\, \text{Ric} = \bar R_{ij}(x)y^i y^j, \,\,\,\,\,\, R_{ij}= \tfrac{1}{2}\left(\bar R_{ij}(x) + \bar R_{ji}(x)\right)
\end{align}
in terms of the curvature tensor and Ricci tensor of the affine connection, given by
\begin{align}\label{eq:affine_curvatures}
    \bar R^k{}_{\ell ij} = \partial_i\Gamma^k_{j\ell} - \partial_j\Gamma^k_{i\ell} + \Gamma^k_{im}\Gamma^m_{j\ell} - \Gamma^k_{jm}\Gamma^m_{i\ell}, \qquad \bar R_{lk} = \bar R^i{}_l{}_{ik},
\end{align}
respectively. We will refer to $\bar R^k{}_{\ell ij}$ and $\bar R_{lk}$ as (the components of) the \textit{affine curvature tensor} and \textit{affine Ricci tensor}, respectively. Since the last identity in \eqref{eq:symm_ricci_Berwald} will prove to be of particular interest later, we restate it as a lemma.
\begin{lem}
\label{lem:RicciTensors}
The Finsler-Ricci tensor of a Berwald space is the symmetrization of the affine Ricci tensor.
\end{lem}
\noindent We say that a Berwald space is \textit{affinely Ricci-flat} if $\bar R_{ij}=0$. Note that this is a stronger condition than (standard) Ricci-flatness, $R_{ij}=0$.

\subsection{Metrizability}

\noindent Given a Finsler space of Berwald type, the canonical nonlinear connection reduces to a linear connection on $TM$ that is torsion-free. This leads to the natural question of whether this linear connection is `metrizable'.

\begin{defi}
    Let $F$ be a Berwald metric with underlying manifold $M$. We say that (the canonical affine connection of) $F$ is metrizable if it arises as the Levi-Civita connection of some pseudo-Riemannian metric on $M$.
\end{defi}

\noindent For readability, we usually omit the phrase `the canonical affine connection of' in between brackets and we simply say that $F$ is metrizable. A slightly weaker property is that of local metrizability.  

\begin{defi}
    A Berwald space with underlying manifold $M$ is said to be locally metrizable if each point in $M$ has a neighborhood that admits a pseudo-Riemannian metric whose Levi-Civita connection coincides with the affine connection restricted to that neighborhood.
\end{defi}

\section{$m$-Kropina metrics}\label{sec:VGR}
\noindent An $m$-Kropina metric (also called generalized Kropina metric, Bogoslovsky metric or Bogoslovsky-Kropina metric) is a Finsler metric of the form 
\begin{align}\label{eq:mkrop}
    F = \alpha^{1+m}\beta^{-m},
\end{align}
where the building blocks
\begin{align}
    \alpha = \sqrt{\left|a_{ij}y^iy^j\right|},\qquad \beta = b_iy^i
\end{align}
are constructed from a \mbox{(pseudo-)Riemannian} metric $a = a_{ij}(x)\D x^i\D x^j$ and a 1-form $b = b_i(x)\D x^i$, respectively, and where $m$ is a real parameter. By a slight abuse of terminology one also refers to $\alpha$ and $\beta$ simply as the \mbox{(pseudo-)Riemannian} metric and the 1-form, respectively. We remark that it follows from our definition of a Finsler metric that \bluetext{(excluding the trivial pseudo-Riemannian case $m=0$)} the 1-form $b$ is nowhere vanishing, which we will henceforth take for granted. When $m=1$ the $m$-Kropina metric reduces to the standard Kropina metric \cite{Kropina} $F = \alpha^2/\beta$.

In the physics literature, spacetimes with a metric of $m$-Kropina type have also been dubbed \textit{Very General Relativity} (VGR) spacetimes \cite{Fuster:2018djw} or \textit{General Very Special Relativity} (GVSR) spacetimes \cite{Kouretsis:2008ha}, introduced as generalizations of \textit{Very Special Relativity} (VSR) \cite{Cohen:2006ky,Gibbons:2007iu} spacetimes. Indeed, as an important application, $m$-Kropina metrics have been found to yield exact solutions\cite{Heefer_2023_Finsler_grav_waves,Fuster:2018djw,Fuster:2015tua} to the Pfeifer and Wohlfarth's Finslerian extension of Einstein's field equation\cite{Pfeifer:2011xi,Pfeifer:2013gha,Hohmann_2019}; see (Heefer, 2024)\cite{HeeferPhdThesis} for a recent review.

We further introduce the notation 
\begin{align}
    |b|^2 \equiv a_{ij}b^ib^j
\end{align}
for the squared norm of $\beta$ with respect to $\alpha$. Also, throughout the remainder of this article, all indices are raised and lowered with $a_{ij}$, and $\mathring\nabla$ denotes the Levi-Civita connection of $a_{ij}$. 

\subsection{The Berwald condition}\label{sec:berwald_condition}
\noindent As shown recently\cite{Heefer2023mKropNull} (see also (Heefer, 2024)\cite{HeeferPhdThesis} for a more in-depth treatment and an alternative proof), an m-Kropina metric $F = \alpha^{1+m}\beta^{-m}$ with $n=\dim M >2$ is of Berwald type if and only if there exists a smooth vector field with components $f^i$ on $M$ such that
\begin{align}\label{eq:C-mtr_mKrop_betrald_cond}
\mathring\nabla_j b_i = m (f_k b^k)a_{ij} + b_i f_j  - m f_i b_j.
\end{align}
Here and throughout the remainder of the article, $\mathring\nabla$ denotes the Levi-Civita connection of $a_{ij}$. This Berwald condition agrees with the one obtained by Matsumoto\cite{handbook_Finsler_vol2_matsumoto}, which was proven under slightly different assumptions.


Whenever the condition \eqref{eq:C-mtr_mKrop_betrald_cond} is satisfied, the affine connection coefficients of the Berwald $m$-Kropina space can be expressed in terms of the Christoffel symbols $\mathring\Gamma^k_{ij}$ of the Levi-Civita connection corresponding to $\alpha$ as follows\cite{HeeferPhdThesis,Heefer2023mKropNull,handbook_Finsler_vol2_matsumoto}
\begin{align}\label{eq:delta_gamma_General}
\Gamma^\ell_{ij} = \mathring\Gamma^\ell_{ij} + \Delta \Gamma^\ell_{ij},\qquad \Delta \Gamma^\ell_{ij} = m a^{\ell k}\left(a_{ij}f_k - a_{jk}f_i - a_{ki}f_j\right)
\end{align}
This implies in particular that for any given data ($\alpha,\beta$) there can be at most one vector field $f^k$ satisfying \eqref{eq:C-mtr_mKrop_betrald_cond}. (This can be seen for instance by considering $\Delta\Gamma^\ell_{ij}$ with $i=j\neq \ell$.)

\section{Characterization of metrizability}
\label{sec:metrizability}

\noindent From here onwards we will focus on $m$-Kropina metrics of Berwald type with $n=\dim M>2$. The main aim in this section is answer the following question: \\

\noindent \textbf{Q: }\textit{Under what conditions can the affine connection on a Berwald $m$-Kropina space be understood, locally, as the Levi-Civita connection of a pseudo-Riemannian metric?}\\

\noindent In \S\ref{sec:metrizability_local}, we obtain necessary and sufficient conditions for local metrizability. Then we discuss several special cases in \S\ref{sec:specialcases}. And finally, in \S\ref{sec:global_metriz}, we have a brief discussion on global metrizability.

\subsection{The main theorem for local metrizability}\label{sec:metrizability_local}

\noindent It is a standard result that the Ricci tensor of any Levi-Civita connection is symmetric. This immediately implies that symmetry of the Ricci tensor constructed from the affine connection of a Berwald space---that is, symmetry of the \textit{affine Ricci tensor} (see Section \ref{sec:Berwald})---is a necessary condition for metrizablilty. To characterize this condition, we will first derive a simple formula for the skew-symmetric part of the affine Ricci tensor, denoted by $\mathcal A(\bar R)$ and having components $\mathcal A(\bar R)_{ij} = \bar R_{[ij]} = \tfrac{1}{2}\left(\bar R_{ij} - \bar R_{ji}\right)$. Moreover, to prevent confusion, we remark that the expression $\D f$ below denotes the exterior derivative of the 1-form $f$ with components $f_i = a_{ij}f^j$ appearing in the Berwald condition \eqref{eq:C-mtr_mKrop_betrald_cond}. We also recall that $n = \dim M$.
\begin{lem}\label{lem:ricci_general}
The skew-symmetric part of the affine Ricci tensor is given by 
\begin{align}\label{eq:lem_skew_R_formula}
    \mathcal A(\bar R) = -\frac{mn}{2} \D f
\end{align}
\end{lem}
\begin{proof}
From the definition \eqref{eq:affine_curvatures} of the affine Ricci tensor of a Berwald space and the fact that $\Gamma^k_{ij} = \Gamma^k_{ji}$ it follows that the skew-symmetric part of the affine Ricci tensor can be written as
\begin{align}
\bar R_{[ij]} \equiv \tfrac{1}{2}\left(\bar R_{ij} - \bar R_{ji}\right)  = \partial_{[i}\Gamma^k_{j]k}.
\end{align}
Since the Levi-Civita connection $\mathring\Gamma^k_{ij}$ has a symmetric Ricci tensor, it follows that $\partial_{[i}\mathring\Gamma^k_{j]k}=0$ and hence we have 
\begin{align}\label{eq:add_eq_}
    \bar R_{[ij]} &= \partial_{[i}\Gamma^k_{j]k} = \partial_{[i}\Delta \Gamma^k_{j]k}=\tfrac{1}{2}\left(\partial_{i}\Delta\Gamma^k_{kj}-\partial_{j}\Delta\Gamma^k_{ki}\right).
\end{align}
Employing the expression \eqref{eq:delta_gamma_General} for $\Delta \Gamma^k_{ij}$, we find that
\begin{align}\label{eq:fsmooth}
    \Delta\Gamma^k_{kj} =  - m n f_j
\end{align}
and this yields
\begin{align}
\bar R_{[ij]} = \frac{mn}{2}(\partial_jf_i-\partial_if_j),
\end{align}
which is precisely the coordinate expression of \eqref{eq:lem_skew_R_formula}, as desired.
\end{proof}
We then have the following characterization of local metrizability. 
\begin{theor}\label{theor:metrizability_main}
For an $m$-Kropina metric  $F = \alpha^{1+m}\beta^{-m}$ of Berwald type with $\dim M>2$, the following are equivalent:
\begin{enumerate}[(i)]
    \item $F$ is locally metrizable;
    \item The affine Ricci tensor is symmetric, $\bar R_{ij} = \bar R_{ji}$;
    \item \bluetext{The 1-form $f_i$ appearing in \eqref{eq:C-mtr_mKrop_betrald_cond} closed.}
\end{enumerate} 
In this case we can locally write  $f_i =\partial_i\psi$ for some $\psi\in C^\infty(M)$ and we may define $\tilde a_{ij} = e^{-2m\psi}a_{ij}$ and $\tilde b_i = e^{-(1+m)\psi}b_i$ and define $\tilde \alpha$ and $\tilde \beta$ accordingly. Then the following statements hold:
\begin{itemize}
    \item $F$ can be written as $F = \tilde\alpha^{1+m}\tilde\beta^{-m}$;
    \item $\tilde \beta$ is parallel with respect to $\tilde \alpha$;
    \item $\tilde a_{ij}$ provides a local metrization of $F$.
\end{itemize}
\end{theor}
\begin{proof}
    We will prove that (i)$\Rightarrow$(ii)$\Rightarrow$(iii)$\Rightarrow$(i). The first two implications are rather straightforward. 

    (i)$\Rightarrow$(ii): \bluetext{If $F$ is locally metrizable then locally, the affine connection is just a Levi-Civita connection in disguise. Since any Levi-Civita connection has a symmetric Ricci tensor, it follows trivially that the affine Ricci tensor of $F$ is symmetric. }
    
    (ii)$\Rightarrow$(iii): If the Ricci tensor is symmetric then Lemma \ref{lem:ricci_general} shows that $\D f =0$. 
    
    The remaining nontrivial part of the proof amounts to showing that (iii) implies (i). 
    
    (iii)$\Rightarrow$(i): If $f_i$ is closed then $f$ is locally exact, i.e. $f_i = \partial_i\psi$ for some (locally defined) smooth function $\psi$ on $M$. Consider the following transformed metric and 1-form, $\tilde a_{ij} = e^{-2m\psi} a_{ij}$ and $\tilde b_i = e^{-(1+m)\psi}$ and define $\tilde \alpha$ and $\tilde\beta$ accordingly. Then $F = \alpha^{1+m}\beta^{-m}=\tilde\alpha^{1+m}\tilde\beta^{-m}$ and it is straightforward to show that the Christoffel symbols of $\tilde a_{ij}$ are given by
    \begin{align}
        \tilde\Gamma^k_{ij} = \mathring\Gamma^k_{ij} - m\left(f_i\delta^k_j+f_j\delta^k_i - f^ka_{ij}\right)
    \end{align}
    in terms of those of $a_{ij}$. Hence if we denote by $\tilde\nabla$ the Levi-Civita connection of $\tilde \alpha$ then we find that
    \begin{align}
        \tilde\nabla_j\tilde b_i = e^{-(1+m)\psi}\left[\mathring\nabla_j b_i -( m (f_k b^k)a_{ij} + b_i f_j  - m f_i b_j)\right] = 0,
    \end{align}
    which vanishes, by \eqref{eq:C-mtr_mKrop_betrald_cond}. In other words, $\tilde b_i$ is parallel w.r.t. $\tilde a_{ij}$. Since we can write $F$ as $F = \tilde\alpha^{1+m}\tilde\beta^{-m}$, it follows from Prop. \ref{prop:ab_cc_pre} below that the affine connection \bluetext{of} $F$ is just the Levi-Civita connection of $\tilde a_{ij}$. Hence $F$ is locally metrizable, namely by $\tilde a_{ij}$, completing the proof. 
\end{proof}
The mapping $(\alpha,\beta)\mapsto (\tilde\alpha,\tilde\beta)$ is a generalization of what is called the $f$\textit{-change} in \mbox{(Matsumoto, 2003)\cite{handbook_Finsler_vol2_matsumoto}}. In the proof we have made use of the following well-known result; see e.g. (Heefer, 2024)\cite{HeeferPhdThesis}.

\begin{prop}\label{prop:ab_cc_pre}
    Let $F=\alpha\phi(\beta/\alpha)$ be an $(\alpha,\beta)$-metric and suppose that $\mathring\nabla_i b_j=0$. Then $F$ is Berwald and the affine connection of $F$ coincides with that of $\alpha$.
\end{prop}

\subsection{Special types of 1-forms $b$}\label{sec:specialcases}

\noindent Having established the general criteria for local metrizability in Theorem \ref{theor:metrizability_main}, we now turn to several specific possibilities with respect to the causal character of the 1-form $b$:
\begin{enumerate}
    \item $b$ is nowhere null;
    \item $b$ is null and closed;
    \item $b$ is null but not necessarily closed.
\end{enumerate}

\subsubsection{$b$ is nowhere null}
\noindent In this case, we can show that $f_i$ is necessarily closed, and hence our theorem guarantees local metrizability. This case was already known to Matsumoto\cite{handbook_Finsler_vol2_matsumoto}. To see that $f_i$ is closed, note that contracting \eqref{eq:C-mtr_mKrop_betrald_cond} with $b^i$ leads to
\begin{align}
    \tfrac{1}{2}\mathring\nabla_j\left(|b|^2\right) = |b|^2f_j,
\end{align}
from which it follows, as long as $b$ is nowhere null, that 
\begin{align}\label{eq:f_in_terms_of_b}
    f_j =\partial_j\ln\left(\sqrt{||b|^2|}\right).
\end{align}

\noindent Hence we have the following corollaries.

\begin{cor}\label{prop:non-nullmetrizable}
    An $m$-Kropina space with $|b|^2\neq 0$ is locally metrizable.
\end{cor}
\begin{cor}\label{prop:posdefmetrizable}
    An $m$-Kropina space with positive definite $\alpha$ is locally metrizable.
\end{cor}

\noindent \bluetext{In fact, it turns out that in these cases we even have global metrizability. This will be discussed in Section \ref{sec:global_metriz}.}

\subsubsection{$b$ is null and closed}
\noindent This case was the subject of our earlier paper\cite{Heefer2023mKropNull}. We restate the main relevant result here.

\begin{prop}\label{prop:main_previous}
If $b$ is null and closed then $F$ is locally metrizable if and only if around each point in $M$ there exist coordinates $(u,v,x^3,\dots,x^n)$ such that {\normalfont$b = \D u$} and
\begin{align}\label{eq:form_of_a}
a = -2\D u\D v + \left[\tilde H(u,x) + \rho(u)v\right]\D u^2 + 2W_a(u,x)\D u\D x^a + h_{ab}(u,x)\D x^a\D x^b,
\end{align}
where $a,b = 3,\dots n$ and $\tilde H, \rho, W_a, h_{ab}$ are smooth functions.

In this case, the affine connection restricted to the chart corresponding to the coordinates $(u,v,x^3,\dots x^n)$ is metrizable by the following pseudo-Riemannian metric:
\begin{align}\label{eq:metrizing_metric}
{\normalfont \tilde a = e^{\frac{m}{1-m}\int^u \rho(\tilde u)\D \tilde u } a}
\end{align}
\end{prop}


\subsubsection{$b$ is null but not necessarily closed}
\noindent Finally, if $b$ is null but not necessarily closed then the local metrizability of $F$ can be characterized in a way similar to that of Prop. \ref{prop:main_previous}, as we will show next.

\begin{prop}\label{prop:null_metrization_classification}
If $|b|^2=0$ then $F$ is locally metrizable if and only if around each point in $M$ there exist coordinates $(u, v, x^3, \dots, x^n)$ such that
\begin{align}
    a &= e^{2m\psi}\left(-2\D u\D v + H\D u^2 + 2W_a\D u\D  x^a + h_{ab}\D x^a \D x^b\right), \label{eq:pra}\\ 
    b &= e^{(1+m)\psi}\D u, \label{eq:prb}
\end{align}
where $\psi$ is any \bluetext{locally defined} smooth function on $M$, $a,b = 3,\dots n$ and $H, W_a, h_{ab}$ are \bluetext{local} smooth functions depending only on $u$ and $x^a$. 

In this case, $F$ can be written as
\begin{align}\label{eq:standard_pp_wave_m_Kropina}
    F = \left|-2\D u\D v + H\D u^2 + 2W_a\D u\D  x^a + h_{ab}\D x^a \D x^b\right|^{(1+m)/2}\left(\D u\right)^{-m}
\end{align}
and the metric $-2\D u\D v + H\D u^2 + 2W_a\D u\D  x^a + h_{ab}\D x^a \D x^b$ provides a local metrization of $F$.
\end{prop}
\begin{proof}
We first prove the `only if' part. By Theorem \ref{theor:metrizability_main}, if $F$ is locally metrizable then there exist $\tilde a_{ij} = e^{-2m\psi}a_{ij}$ and $\tilde b_i = e^{-(1+m)\psi}b_i$ with corresponding $\tilde \alpha$ and $\tilde \beta$ such that $F = \tilde\alpha^{1+m}\tilde\beta^{-m}$ and such that $\tilde \beta$ is parallel with respect to $\tilde a_{ij}$. It is a well-known result that any pseudo-Riemannian metric $\tilde a_{ij}$ admitting a nowhere vanishing parallel null 1-form $\tilde b$, admits coordinates $(u,v,x^3,\dots x^n)$ around any point such that 
\begin{align}
    \tilde a &= -2\D u\D v + H\D u^2 + 2W_a\D u\D  x^a + h_{ab}\D x^a \D x^b, \\
    \tilde b &= \D u,
\end{align}
where the functions $H, W_a, h_{ab}$ depend only on $u$ and $x^a$, and where $a = 3,4,\dots n$. It is immediately verified that $a,b,$ and $F$ now attain the desired form, completing the `only if' direction of the proof. 

Conversely, if $a$ and $b$ have the stated form, then $F$ is given by \eqref{eq:standard_pp_wave_m_Kropina} in terms of $\tilde a$ and $\tilde b$, and since $\tilde b=\D u$ is parallel with respect to $\tilde \alpha$, we must have $f_k=0$ in the Berwald condition \eqref{eq:C-mtr_mKrop_betrald_cond}. It thus follows from \eqref{eq:delta_gamma_General} that $\tilde \alpha$ provides a local metrization of $F$, completing the proof.
\end{proof}
To end this section, we note that as a special case of Prop. \ref{prop:null_metrization_classification} we recover Prop.  \ref{prop:main_previous}, thus providing an alternative proof of the result.

\begin{proof}[Proof of Prop. \ref{prop:main_previous} from Prop. \ref{prop:null_metrization_classification}]
    Suppose first that $F$ is locally metrizable, so that $a$ and $b$ are given by \eqref{eq:pra} and \eqref{eq:prb}, where the metric functions do not depend on the coordinate $v$. The fact that the 1-form $b = e^{(1+m)\psi}\D u$ is closed implies that $\psi = \psi(u)$. In this case, it can be verified that the coordinate transformation given by $\bar u = \int^u e^{(1+m)\psi(u)}\D u$ and $\bar v = e^{(m-1)\psi(u)}v$ brings the 1-form into the form $b = \D\bar u$ (as desired) and the metric into the form
    \begin{align}
        a = -2\D \bar u\D \bar v + \bar H\D \bar u^2 + 2\bar W_a\D \bar u\D  x^a + \bar h_{ab}\D x^a \D x^b,
    \end{align}
    where
    \begin{align}
        \bar H = e^{-2\psi}\left[H + 2(m-1) \psi'v\right],\quad \bar W_a =  e^{(m-1)\psi} W_a, \quad \bar h_{ab} &= e^{2m\psi} h_{ab}.
    \end{align}
    Identifying $\tilde H(\bar u,x^a) = e^{-2\psi(u(\bar u))}H$ and $\rho(\bar u) = 2e^{-(1+m)\psi(u(\bar u))}(m-1)\psi'(u(\bar u))$ it is easily verified that $a$ also attains the desired form \eqref{eq:form_of_a} in the new coordinates. This proves the implication to the right.
    
    Conversely, if $a$ and $b$ have the stated form, then we may define $\psi(u)$ such that $\rho(u) = 2e^{-(1+m)\psi}(m-1)\psi' = 2(m-1)d\psi/d\bar u$, by setting $\psi(u) = \tfrac{1}{2(m-1)}\int^{\bar u(u)}\rho(u)\D u$, and then we can simply perform the same coordinate transformation backward. It is clear that this results in a metric and 1 form of the form \eqref{eq:pra} and \eqref{eq:prb} and hence, by Prop. \ref{prop:null_metrization_classification}, $F$ is locally metrizable, proving the implication to the left. Furthermore, the local metrization of Prop. \ref{prop:null_metrization_classification} given by $-2\D u\D v + H\D u^2 + 2W_a\D u\D  x^a + h_{ab}\D x^a \D x^b$ turns under this transformation into \eqref{eq:metrizing_metric}, showing that \eqref{eq:metrizing_metric} provides a local metrization for $F$, completing the proof.
\end{proof}

\subsection{Global metrizability}\label{sec:global_metriz}

\noindent Note that in the proof of Theorem \ref{theor:metrizability_main} we made use of the fact that $\D f=0$ implies that $f$ is a locally exact 1-form. Here $f$ is the 1-form appearing in the Berwald condition \eqref{eq:C-mtr_mKrop_betrald_cond}. In the special case where $f$ is (globally) exact, the proof guarantees that the metric $\tilde a$ providing the metrization is in fact globally defined, implying that the Finsler space under consideration is globally metrizable. Hence we have the following.

\begin{prop}\label{prop:globmetr}
If $f$ is an exact 1-form then $F$ is (globally) metrizable.
\end{prop}

\noindent In general, not every closed 1-form is (globally) exact, and hence the proof of the theorem does not tell us whether local metrizability is equivalent to global metrizability. \bluetext{In certain specific cases we can say more, however. First, in the case that the 1-form $\beta$ is nowhere null, we can use Prop. \ref{prop:globmetr} to strengthen Cor. \ref{prop:non-nullmetrizable} and Cor. \ref{prop:posdefmetrizable}.}

\begin{cor}\label{prop:non-nullmetrizable_global}
    \bluetext{An $m$-Kropina space with $\beta$ nowhere null is (globally) metrizable.}
\end{cor}
\begin{proof}
    \bluetext{Since $||b|^2|$ is a globally defined function, it follows from \eqref{eq:f_in_terms_of_b} that in this case, $f$ is globally exact. Hence the result follows from Prop. \ref{prop:globmetr}.}
\end{proof}
\begin{cor}\label{prop:posdefmetrizable_global}
    \bluetext{An $m$-Kropina space with positive definite $\alpha$ is (globally) metrizable.}
\end{cor}

\noindent \bluetext{These results are independent of the topology of the underlying manifold but do require a nontrivial assumption about $\beta$. The following theorem shows that if, alternatively, we know more about the topology, we can also obtain some interesting results even when $\beta$ is not necessarily nowhere vanishing.}

\begin{theor}\label{theor:metrizability_global}
Suppose the first de Rham cohomology group of $M$ is trivial. Then the following are equivalent:
\begin{enumerate}[(i)]
    \item $F$ is metrizable;
    \item $F$ is locally metrizable;
    \item The affine Ricci tensor is symmetric;
    \item $f_i$ is a closed 1-form.
\end{enumerate} 
In that case, writing \bluetext{$f =\text{d}\psi$, a global metrization of $F$ is given by $\tilde a = e^{-2m\psi}a$ and moreover, $\tilde b = e^{-(1+m)\psi}b$ is parallel with respect to $\tilde a$.}
\end{theor}
\begin{proof}
    We only need to prove the equivalence of global metrizability with the other three (already equivalent) statements. Since global metrizibility trivially implies local metrizability, it suffices to show that $f$ being closed implies that $F$ is metrizable. So suppose that $f$ is closed. \bluetext{Then since the first de Rham cohomology group is trivial, $f$ is exact, $f = \D\psi$, with $\psi$ defined globally. The proof of Theorem \ref{theor:metrizability_main} shows that $\tilde a = e^{-2m\psi}a$ provides a global metrization. Hence $F$ is globally metrizable.}
\end{proof} 

Since any simply connected manifold has a trivial first de Rham cohomology group, we have, in particular, the following result.

\begin{cor}
    If $M$ is simply connected then (i)-(iv) are equivalent.
\end{cor}

\noindent The remaining question is the following one:\\

\textbf{Q:} \textit{Do locally but not globally metrizable $m$-Kropina spaces exist at all?}\\

\noindent This is still an open question and we defer its analysis to future work, even though it seems likely that the answer should be affirmative. It is clear, at least, that such a space would need to have a 1-form $f$ appearing in the Berwald condition that is \textbf{closed but not exact}. For if it were not closed, the space would not be locally metrizable and hence not metrizable; and if it were exact, then it would be globally metrizable. A straightforward method for trying to produce a counterexample is thus given as follows. 
\begin{enumerate}
    \item \bluetext{Choose your favorite manifold $M$ with nontrivial first de Rham cohomology group.}
    \item Pick a 1-form $f$ that is closed but not exact.
    \item Pick a pseudo-Riemannian metric $\alpha$.
    \item Try to solve the Berwald condition for $\beta$.
\end{enumerate}
An obvious first guess in dimension 3 would be to let $M = \R^2\times S^1$ with the metric $\D s^2 = \D x^2+\D y^2+\D\theta^2$ and $f = \D\theta$, where $(x,y)$ are the standard coordinates on $\R^2$ and $\theta$ is the standard angle coordinate on the circle $S^1$. The 1-form $f = \D\theta$ is the prime example of a closed but not exact 1-form, of course. In this case, which is relatively computable, the system defined by the Berwald condition turns out not to be solvable, however. That is, there is no 1-form $\beta$ that leads to the prescribed $f = \D\theta$. The same is true if one instead uses the metric $\D s^2 = \D x^2 +\D y^2+y^2\D\theta^2$ \bluetext{on $\R^2\times S^1$}, which corresponds to the standard flat metric $\D s^2 = \D x^2 + \D \tilde y^2 + \D z^2$ on $\R^3$ \bluetext{if we view $\R^2\times S^1$ as an open subset of $\R^3$ in the natural way,} where $y$ is basically the radial coordinate in the $\tilde y z$-plane. Hence, in order to produce counterexamples---if these exist (which we do expect)---one needs to consider more elaborate examples of manifolds with closed but not exact 1-forms. We leave this to future study.

\color{black}
\section{Ricci-flatness of \texorpdfstring{$m$}{m}-Kropina spaces in (1+3)D}\label{sec:classification_ricci_flat}

\noindent We now turn to Ricci-flatness. Recall that a Finsler space is said to be Ricci-flat if the Finsler-Ricci tensor vanishes, $R_{ij}=0$, or equivalently, if the Ricci scalar vanishes, $\text{Ric}=0$. We restrict ourselves in this section to the case where $M$ is four-dimensional with $a_{ij}$ having Lorentzian signature. This is the case that is most relevant to applications in gravitational physics. We'll refer to such metrics as \mbox{(1+3)D} $m$-Kropina metrics.

\subsection{Ricci-flat nowhere-null \texorpdfstring{$m$}{m}-Kropina metrics in \mbox{(1+3)D}}

\noindent First, we characterize Ricci-flatness of $m$-Kropina metrics in \mbox{(1+3)D} with \textit{nowhere-null} 1-form, $|b|^2\neq 0$.
\begin{prop}\label{prop:Ricci_flat_F}

    A \bluetext{Berwald} $(1+3)$D $m$-Kropina space with $|b|^2\neq 0$ is Ricci-flat if and only if it is locally of the form
    \begin{align}\label{eq:unique_non-null_ricci_flat_mKrop}
        F = \left|\eta_{ij}\D x^i \D x^j\right|^{(1+m)/2}\left(c_i\D x^i\right)^{-m}
    \end{align}
    with $[\eta_{ij}]=$ diag$(-1,1,1,1)$ and $c_i=$ const. 
\end{prop}
\begin{proof}
We first prove the `if' direction, i.e. we assume $F$ is of the form \eqref{eq:unique_non-null_ricci_flat_mKrop} and prove that $F$ is Ricci-flat. The 1-form $c_i\D x^i$ is parallel with respect to $\eta_{ij}$, and hence it follows by Prop. \ref{prop:ab_cc_pre} that $F$ has the same affine connection as $\eta$, which is Ricci-flat. Hence the  Ricci tensor constructed from the affine connection of $F$---the affine Ricci tensor---vanishes as well, i.e. $F$ is affinely Ricci-flat. This implies, in particular, that $F$ is Ricci-flat, by Lemma \ref{lem:RicciTensors}. 

Conversely, to prove the `only if' direction, suppose that $F$ is Ricci-flat. \bluetext{Note also that $F$ is metrizable, by Cor. \ref{prop:non-nullmetrizable_global}.} Hence, defining $\tilde \alpha,\tilde \beta$ as in Theorem \ref{theor:metrizability_main}, we may write $F = \tilde\alpha^{1+m}\tilde\beta^{-m}$ and in this case $F$ has the same affine connection as $\tilde \alpha$, so since $F$ is Ricci-flat, $\tilde \alpha$ is Ricci-flat. But by Theorem \ref{theor:metrizability_main}, $\tilde \alpha$ admits a non-null parallel 1-form, namely $\tilde b$, so if $\tilde \alpha$ is Ricci-flat it must actually be flat, by Lemma \ref{lem:EFE_sol_with_parallel_field} below. Hence, in suitable coordinates, it can be written as $\tilde a = \eta_{ij}\D x^i \D x^j$. Since $\tilde \beta$ is parallel with respect to this metric, it must have constant coefficients in these coordinates, and \eqref{eq:unique_non-null_ricci_flat_mKrop} follows, completing the proof.
\end{proof}

In the proof we have made use of the following well-known result; see e.g. (Heefer, 2024)\cite{HeeferPhdThesis} for a self-contained proof.
\begin{lem}\label{lem:EFE_sol_with_parallel_field}
    Any Ricci-flat Lorentzian metric admitting a non-null parallel vector field or 1-form is flat.
\end{lem}
\noindent We remark that the \textit{representation} of $F$ in terms of a specific choice of $\alpha$ and $\beta$ is, of course, not unique (we can multiply both $\alpha$ and $\beta$ by suitably related conformal factors without changing $F$), and the proposition does not imply that $\alpha$ has to be equal to the Minkowski metric $\eta$. However, since any representative $\alpha$ is related to the $\tilde\alpha$ from the proposition by a conformal transformation (\bluetext{in the Lorentzian case this is most easily seen by looking at the zeros of $F = \tilde\alpha^{1+m}\tilde\beta^{-m} = \alpha^{1+m}\beta^{-m}$ and using the well-known result that two Lorentzian metrics with identical light cones must be conformally related}), we have the following \bluetext{corollary}.
%
%
%
\begin{cor}
    If $F$ is Berwald and Ricci-flat and $|b|^2\neq 0$ then $\alpha$ is conformally flat.
\end{cor}

\noindent Prop. \ref{prop:Ricci_flat_F} shows that any Berwald, Ricci-flat $m$-Kropina metric with $|b|^2\neq 0$ must be of the form \eqref{eq:unique_non-null_ricci_flat_mKrop}. We end this section with an example of a nowhere-null $m$-Kropina space from the literature and show that it is indeed the trivial Finsler metric \eqref{eq:unique_non-null_ricci_flat_mKrop} in disguise.

\begin{ex}\label{example}
In (Fuster et al., 2018)\cite{Fuster:2018djw} the following $m$-Kropina metric was considered as a cosmological model of spacetime in the context of Finsler gravity---an extension of Einstein's general theory of relativity that allows for Finsler spacetime geometries:
\begin{align}\label{eq:metricex}
    F &= \sqrt{\left|\D t^2 - a(t)^2\left(\D x^2+\D y^2+\D z^2\right)\right|}^{1+m}\left(c\, a(t)^{1/m}\D t\right)^{-m},
\end{align}
The metric is written in terms of the coordinates $(t,x,y,z)$, $a(t)$ is an arbitrary function, and $c$ is a positive constant. We remark that the notation in the original article was somewhat different.\footnote{A slightly different framework is used in (Fuster et al., 2018)\cite{Fuster:2018djw}, where the geometry is in fact characterized by an $r$-homogeneous Lagrangian (for arbitrary $r$) rather than a Finsler metric. But the metric \eqref{eq:metricex} displayed here represents the same geometry. It is related to the Lagrangian in (Fuster et al., 2018)\cite{Fuster:2018djw} by $L = F^\frac{2}{1+m}$. Similarly, our parameter $m$ is related to the parameter $n$ in (Fuster et al., 2018)\cite{Fuster:2018djw} by $n = \frac{-2m}{1+m}$.} This metric is Ricci-flat (and hence a solution to most candidate field equations in Finsler gravity).

Since the 1-form $b = c\, a(t)^{1/m}\D t$ is nowhere null, Prop. \ref{prop:Ricci_flat_F} dictates that the metric must be of the trivial form \eqref{eq:unique_non-null_ricci_flat_mKrop} in suitable coordinates. It indeed turns out (and it is easy to check) that the coordinate transformation 
\begin{align}
\tilde t = \int^t\frac{c^{-m}}{a(t)}\D t,\qquad \tilde x = c^{-m} x, \qquad \tilde y = c^{-m} y, \qquad \tilde z = c^{-m} z,
\end{align}
brings the metric into the trivial form 
\begin{align}
    F=\sqrt{\left|\D \tilde t^2 - \D \tilde x^2-\D \tilde y^2-\D \tilde z^2\right|}^{1+m}\left(\D \tilde t\,\right)^m.
\end{align}
\end{ex}

\subsection{Ricci-flat null \texorpdfstring{$m$}{m}-Kropina metrics in \mbox{(1+3)D}}

\noindent While an $m$-Kropina space with $|b|^2\neq 0$ is always locally metrizable, as we have seen in Prop. \ref{prop:non-nullmetrizable}, this is not always the case when the 1-form is null. With regard to Ricci-flatness, it turns out that in the $|b|^2=0$ case we have to add the additional assumption that $F$ is locally metrizable in order to obtain nice results. Alternatively, one may choose to replace everywhere the two conditions \textit{locally metrizable and Ricci-flat} by the single condition \textit{affinely Ricci-flat}, the latter meaning that the \textit{affine} Ricci tensor vanishes (rather than the Finsler-Ricci tensor, which is its symmetrization, by Lemma \ref{lem:RicciTensors}). Affine Ricci-flatness is not a very common notion in the literature, though, so we will stick to locally metrizable and Ricci-flat throughout the article.
\begin{prop}
    Let $F$ be an $m$-Kropina metric of Berwald type. Then 
    \begin{align}
        \text{F is affinely Ricci-flat}\quad\Leftrightarrow\quad \text{F is locally metrizable and Ricci-flat}. \nonumber
    \end{align}
\end{prop}
\begin{proof}
    If $F$ is affinely Ricci-flat then it is trivially Ricci-flat, since the Finsler-Ricci tensor is the symmetrization of the affine Ricci tensor. In that case, $F$ \bluetext{is} also locally metrizable, by Theorem \ref{theor:metrizability_main}, since the (vanishing) affine Ricci tensor is, in particular, symmetric. For the converse implication, if $F$ is locally metrizable then, by Theorem \ref{theor:metrizability_main}, its affine Ricci tensor is symmetric, and hence the affine Ricci tensor coincides with the Finsler-Ricci tensor. The vanishing of the Finsler-Ricci tensor thus implies, in this case, the vanishing of the affine Ricci tensor.
\end{proof}
To obtain the $|b|^2=0$ analog of Prop. \ref{prop:Ricci_flat_F}, characterizing Ricci-flatness, we will employ the following result proven in (Heefer et al., 2023)\cite{Heefer_2023_Finsler_grav_waves}, an alternative account of which---more in the classical \textit{definition}-\textit{theorem}-\textit{proof} style---can be found in (Heefer, 2024)\cite{HeeferPhdThesis}.
\begin{lem}\label{lem:coordchange}
    Let a Lorentzian metric $a$ and 1-form $b$ on a $4$-dimensional manifold be locally given by
    \begin{align}
        a &= -2\D u \D v + H(u,x)\, \D u^2 + 2W_a(u,x)\,\D x^a\D u +h_{ab}(u,x) \D x^a \D x^b, \label{eq:original_pp_wave_metricl} \\
        b &= \D u.\label{eq:original_1forml}
    \end{align}
    If $a$ is Ricci-flat then $a$ and $b$ can be expressed locally, in suitable coordinates, as
    \begin{align}\label{eq:pp_waves_final_form_body2s}
        a &= -2\D \bar u\D \bar v + \bar H(\bar u,\bar x,\bar y)\D \bar u^2 + \D \bar x^2 + \D \bar y^2,\\
        b &= \D \bar u,
\end{align}
such that $\bar H$ satifies $\partial^2_{\bar x} \bar  H + \partial^2_{\bar y} \bar H=0$.
\end{lem}
Below, $\delta_{ab}$ denotes the Kronecker symbol: $
\delta_{ab} = 1$ iff $a=b$ and otherwise $
\delta_{ab}=0$.
\begin{prop}\label{prop:poep}
    Any locally metrizable $m$-Kropina space in $(1+3)$D with $|b|^2=0$ is Ricci-flat if and only if it is locally of the form
\begin{align}\label{eq:nullmKroptrivialRicciFlat}
    F = \left|-2\D u\D v + H(u,x)\D u^2 + \delta_{ab}\D x^a \D x^b\right|^{(1+m)/2}\left(\D u\right)^{-m}
\end{align}
such that $H$ satifies $\delta^{ab}\partial_a \partial_b H=0$.
\end{prop}
\begin{proof}
First of all, since $\D u$ is parallel, it follows by Prop. \ref{prop:ab_cc_pre} that any $F$ defined by \eqref{eq:nullmKroptrivialRicciFlat} is Ricci-flat, since the Ricci tensor of the Lorentzian metric (which vanishes, as can be checked explicitly) must coincide with the affine Ricci tensor of $F$, which therefore also vanishes, implying that also the Finsler-Ricci tensor vanishes, by Lemma \ref{lem:RicciTensors}. 

Conversely, suppose that $F$ is Ricci-flat. Since $F$ is locally metrizable it can be written in the standard form \eqref{eq:standard_pp_wave_m_Kropina} with $a$ and $b$ given by \eqref{eq:original_pp_wave_metricl} and \eqref{eq:original_1forml}, respectively. Since the 1-form $b = \D u$ is parallel with respect to the corresponding metric $a$, it follows that the affine connection \bluetext{of} $F$ coincides with the \bluetext{Levi-Civita} connection \bluetext{of} $a$. So since $F$ is Ricci-flat, $a$ is Ricci-flat as well and hence it follows from Lemma \ref{lem:coordchange} that $F$ can be written locally in the form \eqref{eq:nullmKroptrivialRicciFlat}, as desired.
\end{proof}
%
%
%
With this, we have classified (locally, at least) all locally metrizable, Ricci-flat $m$-Kropina metrics in $(1+3)$D with constant causal character.

\section{Discussion}

\noindent In this paper, we have studied the local and global metrizability as well as Ricci-flatness of $m$-Kropina spaces of Berwald type. First of all, we obtained necessary and sufficient conditions for the metrizability of Berwald $m$-Kropina metrics $F = \alpha^{1+m}\beta^{-m}$ of arbitrary signature. More precisely, Theorem \ref{theor:metrizability_main} shows that the canonical affine connection of $F$ can locally be obtained as the Levi-Civita connection of some pseudo-Riemannian metric if and only if the Ricci tensor constructed from the affine connection is symmetric. In addition, the theorem gives a third equivalent characterization of local metrizability in terms of the covariant derivative of the defining 1-form $\beta$ with respect to the defining pseudo-Riemannian metric $\alpha$. Using these characterizations, we have classified all locally metrizable $m$-Kropina metrics whose 1-forms have a constant causal character \bluetext{(Prop. \ref{prop:null_metrization_classification} and Cor. \ref{prop:non-nullmetrizable_global})}. These results are consistent with, and extend, the findings obtained in (Heefer et al., 2023)\cite{Heefer2023mKropNull}, where only closed null 1-forms $\beta$ were considered.

In general, global metrizability is a stronger notion than local metrizability. Nevertheless, in the special case where the first de Rham cohomology group of the underlying manifold is trivial (which is true of simply connected manifolds, for instance), we have shown that global metrizability is equivalent to local metrizability. Hence, in this case, our necessary and sufficient conditions for local metrizability also characterize global metrizability. 

Finally, we have classified all Ricci-flat, locally metrizable $m$-Kropina metrics in $(3+1)$D whose 1-forms have a constant causal character. As it turns out, an $m$-Kropina space is locally metrizable and Ricci-flat if and only if it is affinely Ricci-flat---meaning that the Ricci tensor constructed from the canonical affine connection vanishes. We have explicitly constructed \bluetext{(locally)} all affinely Ricci-flat $m$-Kropina spaces of Berwald type \bluetext{(Prop. \ref{prop:poep} and Prop. \ref{prop:Ricci_flat_F})}. 

In part, the results obtained in this paper should be viewed as a small step in furthering our understanding of the pseudo-Riemann metrizability of Berwald spaces in general.\cite{Fuster_2020,Heefer2023mKropNull,voicu2024metrizability} On the other hand, our results have consequences for Finsler extensions of Einstein's general theory of relativity. More precisely, they constrain the possible types of vacuum solutions of Berwald $m$-Kropina type to Pfeifer and Wohlfarth's field equation in Finsler gravity\cite{Pfeifer:2011xi,Pfeifer:2013gha,Hohmann_2019}, as these are typically Ricci-flat.\footnote{In fact, we expect that any vacuum solution of Berwald $m$-Kropina \bluetext{type} is Ricci-flat, although this has only been proven in the case with non-null 1-form\cite{Fuster:2018djw}.} 
In particular, our classification implies that any Ricci-flat vacuum solution of Berwald $m$-Kropina type with non-null 1-form must be trivial. Indeed, as an application, we considered an explicit $m$-Kropina metric from the literature used in cosmology that looks non-trivial at first sight, yet we showed that it is trivial after all. Thus, our results may also be seen as a step toward the complete classification of vacuum solutions of Berwald $m$-Kropina type to Pfeifer and Wohlfarth's field equation.

\section{Acknowledgements}
\noindent S. Heefer wants to thank M. S\'anchez for asking several inspiring questions regarding the difference between local and global metrizability.

\bibliography{VGRmetrizability}

\end{document}